\documentclass[12pt, article]{amsart}
\usepackage{amsmath, amsthm, amscd, amsfonts, amssymb, graphicx, color}
\usepackage[bookmarksnumbered, colorlinks, plainpages]{hyperref}
\hypersetup{colorlinks=true,linkcolor=red, anchorcolor=green, citecolor=cyan, urlcolor=red, filecolor=magenta, pdftoolbar=true}

\newtheorem{theorem}{Theorem}[section]
\newtheorem{lemma}[theorem]{Lemma}
\newtheorem{proposition}[theorem]{Proposition}
\newtheorem{corollary}[theorem]{Corollary}
\theoremstyle{definition}
\newtheorem{definition}[theorem]{Definition}

\theoremstyle{remark}
\newtheorem{remark}[theorem]{Remark}
\numberwithin{equation}{section}

\usepackage[all]{xy}
\allowdisplaybreaks
\newcommand{\ces}[1]{\text{ces}({#1})}
\newcommand{\tand}[1]{\text{tand}({#1})}
\renewcommand{\leq}{\leqslant}
\renewcommand{\geq}{\geqslant}

\begin{document}

\setcounter{page}{1}

\title[]{Closed ideals of operators acting on some families of sequence spaces}

\author[B. Wallis]{Ben Wallis}

%\address{$^{1}$Department of Mathematics, University of AAAA, BBBB 654321, CCCC, India.}
\email{\textcolor[rgb]{0.00,0.00,0.84}{benwallis@live.com}}

%\address{$^{2}$Department of Pure Mathematics, Ferdowsi University of Mashhad, P. O. Box 1159, Mashhad 91775, Iran;
%\newline
%Tusi Mathematical Research Group (TMRG), Mashhad, Iran.}
%\email{\textcolor[rgb]{0.00,0.00,0.84}{second@afa.ac.ir}}

%\dedicatory{This paper is dedicated to Professor ABCD}

%\let\thefootnote\relax\footnote{Copyright 2018 by the Tusi Mathematical Research Group.}



%\date{Received: xxxxxx; Revised: yyyyyy; Accepted: zzzzzz.
%\newline \indent $^{*}$Corresponding author}

\begin{abstract}
We study the lattice of closed ideals in the algebra of continuous linear operators acting on $p$th Tandori and $p'$th Ces\`{a}ro sequence spaces, $1\leqslant p<\infty$, which we show are isomorphic to the classical sequence spaces $(\oplus_{n=1}^\infty\ell_\infty^n)_p$ and $(\oplus_{n=1}^\infty\ell_1^n)_{p'}$, respectively.  We also show that Tandori sequence spaces are complemented in certain Lorentz sequence spaces, and that the lattice of closed ideals for certain other Lorentz and Garling sequence spaces has infinite cardinality.
\end{abstract} \maketitle

%%%%%%%%%%%%%%%%%%%%%%%%%%%%%%%%%%%%%%%%%%%%%%Introduction

\section{Introduction}

We study the lattice of closed ideals in $\mathcal{L}(X)$, the algebra of continuous linear operators acting on a Banach space $X$, where $X$ is chosen from among some sequence spaces defined below.  In doing so, we obtain a partial description of the lattice when $X$ is chosen from the $p$th Tandori sequence space $\tand{p}$ or $p'$th Ces\'{a}ro sequence space $\ces{p'}$, $1\leq p<\infty$.  In particular, the lattice of closed ideals for $\mathcal{L}(\tand{p})$ consists, at least in part, of the chain
\begin{equation}\label{complete}
\{0\}\subsetneq\mathcal{K}\subsetneq\overline{\mathcal{J}}_{\ell_p}\subsetneq\mathcal{L},
\end{equation}
where any further closed ideals must occur between $\overline{\mathcal{J}}_{\ell_p}$ and the unique maximal ideals identified in \cite{Le15,KL16}.  Due to the fact that $\ces{p'}\approx\tand{p}^*$, dual results for $\mathcal{L}(\ces{p'})$ are also obtained.  These descriptions are corollaries of the first main result of the present paper, which states that $\tand{p}\approx(\oplus_{n=1}^\infty\ell_\infty^{2^{n-1}})_p$ (in the natural way).  This extends a result of Astashkin, Le\'{s}nik, and Maligranda, who had shown the same for the nonreflexive case $p=1$ (\cite[Corollary 3]{ALM18}).

Let us also discuss the following spaces.

\begin{definition}
Let $1\leq p,q\leq\infty$.  We set
$$W_{p,q}:=\left(\bigoplus_{n=1}^\infty\ell_q^n\right)_{\ell_p},\;\;\;\text{ and}\;\;\;W_{0,q}:=\left(\bigoplus_{n=1}^\infty\ell_q^n\right)_{c_0}.$$
Following \cite{KL16}, we shall abbreviate $W_p:=W_{p,\infty}$ and $W_0:=W_{0,\infty}$.
\end{definition}

Several recent papers have studied the closed ideals in $\mathcal{L}(W_{p,q})$ for various choices of $p$ and $q$.  Complete descriptions of the lattice of closed ideals in these algebras, strikingly analogous to \eqref{complete}, have been given for $W_{0,2}$ (\cite{LLR04}) and its dual $W_{1,2}$ (\cite{LSZ06}).  Indeed, $\{0\}\subsetneq\mathcal{K}\subsetneq\overline{\mathcal{J}}_{c_0}\subsetneq\mathcal{L}$ describes the lattice for $W_{0,2}$, and \eqref{complete} describes the lattice for $W_{1,2}$.  Of course, it has been known for decades (\cite{Ca41} for the $\ell_2$ case and \cite{GFM60} for the remaining cases) that the algebras of operators acting on $W_{p,p}(=\ell_p)$, $1\leq p<\infty$, and also $W_{0,0}(=c_0)$, each admit precisely three closed ideals:  $\{0\}\subsetneq\mathcal{K}\subsetneq\mathcal{L}$.

Partial results have been obtained for other choices of $p$ and $q$.  As observed in \cite[\S2]{LL05}, when $1<q<\infty$, due to the fact that the spaces $\ell_2^n$ are all uniformly complemented in $\ell_q^{2^n}$, $n\in\mathbb{N}$, we have $W_{0,2}$ and $W_{p,2}$ complemented in $W_{0,q}$ and $W_{p,q}$, respectively.  It follows from the classifications above that $\mathcal{L}(W_{0,q})$ and $\mathcal{L}(W_{p,q})$, $p\neq q$, each admit at least four distinct closed ideals.  Building on results from the excellent paper \cite{LOSZ12}, unique maximal ideals have been identified in the algebras $\mathcal{L}(W_p)$, in \cite{Le15} for the nonreflexive case $p=1$ and in \cite{KL16} for the reflexive cases $1<p<\infty$.  It is also known from \cite[Example 3.9]{LLR04} that the closed ideal structure of $\mathcal{L}(W_p)$ satisfies \eqref{complete}, and that any further closed ideals lie between $\overline{\mathcal{J}}_p$ and the unique maximal ideals from \cite{Le15,KL16}.

When $p=\infty$, the algebras of operators on $W_{\infty,q}$ are surprisingly little understood.  In case $1\leq q<\infty$, the space $W_{\infty,q}$ is isomorphic to $\ell_\infty(\ell_q)$ (cf., e.g., \cite[Remark 3.4]{CM11}), and hence admits uncountably many closed ideals by the main result in \cite{SW15}.  As for the case $W_{\infty,\infty}(=\ell_\infty\approx L_\infty)$, only a few distinct ideals in its algebra of operators have been identified; most of what we know about it is summarized in \cite[\S3]{LL05}.  We do know, however, that $\mathcal{L}(W_{\infty,\infty})=\mathcal{L}(\ell_\infty)$ has a unique maximal ideal $\mathcal{M}_\infty$, which can be described in a number of different ways, for instance as
$$\mathcal{M}_\infty=\mathcal{W}(\ell_\infty)=\mathcal{SS}(\ell_\infty)=\mathcal{E}(\ell_\infty)=\mathcal{X}(\ell_\infty).$$

\begin{remark}Tomasz Kania and Timur Oikhberg have each independently proved (unpublished) that the maximal ideal in $\mathcal{L}(\ell_\infty)$ given above likewise coincides with the finitely strictly singular operators of $\ell_\infty$, i.e. that $\mathcal{M}_\infty=\mathcal{FSS}(\ell_\infty)$.\end{remark}

All notation is standard, such as appears in \cite{LT77,AA02,AK16}, but let us nevertheless remind the reader of some important facts and symbolism.  If $X$ is a Banach space, an (algebraic) {\bf ideal} in $\mathcal{L}(X)$ is any linear subspace $\mathcal{I}$ of $\mathcal{L}(X)$ such that
$$BTA\in\mathcal{I}\;\;\;\forall\; T\in\mathcal{I},\;\;\forall\;A,B\in\mathcal{L}(X).$$
The letters $\mathcal{F}$, $\mathcal{K}$, $\mathcal{W}$, $\mathcal{SS}$, $\mathcal{FSS}$, $\mathcal{X}$, and $\mathcal{E}$ denote the classes of finite-rank, compact, weakly compact, strictly singular, finitely strictly singular, separable range, and inessential operators, respectively.  These all form operator ideals in the sense of Pietsch (see \cite{DJP01} for a discussion of these), and are all closed except for $\mathcal{F}$.  For a Banach space $X$, the zero ideal $\{0\}$ and the entire algebra $\mathcal{L}(X)$ are called {\it trivial} ideals.  When $X$ has a basis, the smallest nontrivial closed ideal is $\mathcal{K}(X)$; in this case, every nonzero closed ideal in $\mathcal{L}(X)$ contains $\mathcal{K}(X)$.

If $T\in\mathcal{L}(E,F)$, where $E$ and $F$ are Banach spaces, we denote by $\mathcal{J}_T$ the class of all Banach space operators factoring through $T$, i.e. for each pair $X$ and $Y$ of Banach spaces we set
$$\mathcal{J}_T(X,Y)=\left\{BTA:A\in\mathcal{L}(X,E),B\in\mathcal{L}(F,Y)\right\}.$$
This is a multiplicative ideal but not necessarily a linear space.  Hence, we define
$$\mathcal{G}_T=\text{span}\left(\mathcal{J}_T\right),$$
the operator ideal generated by $T$.  In case $Id_E\in\mathcal{L}(E)$ is the identity operator on $E$, we write
$$\mathcal{J}_E:=\mathcal{J}_{Id_E}\;\;\;\text{ and }\;\;\;\mathcal{G}_E:=\mathcal{G}_{Id_E}.$$
If $\mathcal{A}$ is a class of operators, denote by $\overline{\mathcal{A}}(X,Y)$ the closure of the set $\mathcal{A}(X,Y)$, and by $[\mathcal{A}](X,Y)$ the closed linear span of $\mathcal{A}(X,Y)$.  We also denote by $\mathcal{S}_E$ the {\bf $\boldsymbol{E}$-strictly-singular} operators, i.e. the set of operators failing to fix a copy of $E$.  Thus,
$$\mathcal{S}_E(X,Y)=\left\{T\in\mathcal{L}(X,Y):T|_F\text{ is not bdd. below for any }F\subset X\text{ s.t. }F\approx E\right\}.$$

The remainder of this paper is organized into sections 2 through 4.  In section 2, we define Lorentz and Garling sequence spaces $d(w,p)$ and $g(w,p)$, and summarize what is known about the closed ideal structure of their operator algebras.  In section 3 we extend the result \cite[Corollary 3]{ALM18} to isomorphically identify the $p$th Tandori sequence space $\tand{p}$ with $W_p$, and the $p'$th Ces\'{a}ro sequence space $\ces{p'}$ with $W_p^*$, for all $1\leq p<\infty$.  Finally, in section 4, show that $W_p$ is complemented in $d(w,p)$ and $g(w,p)$ for $1\leq p<\infty$ and for certain choices of weight $w$, and use these facts to study the closed ideal structures of $\mathcal{L}(\tand{p})$ and $\mathcal{L}(\ces{p'})$.

\section{Closed ideals in $\mathcal{L}(d(w,p))$ and $\mathcal{L}(g(w,p))$}

In \cite{KPSTT12} was introduced the first serious study of the closed ideal structure for the Lorentz sequence space operator algebra $\mathcal{L}(d(w,p))$, $1\leq p<\infty$, $w\in\mathcal{W}$,
$$\mathcal{W}=\left\{(w_i)_{i=1}^\infty\in(0,\infty)^\mathbb{N}\cap c_0\setminus\ell_1:1=w_1\geq w_2\geq w_3\cdots\right\}.$$
In that excellent paper, the authors gave the following partial description.
$$
\xymatrix @R=0pt@C=12pt{ & &  &\mathcal{SS}\ar@{=>}[dr]& & & \\\{0\}\ar@{=>}[r]& \mathcal K\subsetneq\overline{\mathcal{J}}_j\ar[r]&\overline{\mathcal{J}}_{\ell_p}\cap\mathcal{SS}\ar@{:>}[dr]\ar[ur]& & [\mathcal{J}_{\ell_p}+\mathcal{SS}]\ar[r]&\mathcal{S}_{d(w,p)}\ar@{=>}[r]&\mathcal{L}\\ & & &\overline{\mathcal{J}}_{\ell_p}\ar[ur]& & & }
$$
Here, $j:\ell_p\to d(w,p)$ denotes the formal identity (the natural map between canonical bases).  We are suppressing the ``$(d(w,p))$'' in notation like ``$\mathcal{L}(d(w,p))$.''  A single arrow ($\longrightarrow$) means $\subseteq$, but that we do not know whether the inclusion is strict.  A double arrow ($\Longrightarrow$) means a unique immediate successor.  In contrast, a dotted double arrow ($\xymatrix @C=20pt{\ar@{:>}[r]&}$) indicates an immediate successor which may or may not be unique.

Some additional information about the ideal structure of $\mathcal{L}(d(w,p))$ is not represented in the above diagram.  For instance, the authors proved in \cite[Theorem 3.5]{KPSTT12} that $\mathcal{FSS}(d(w,p))=\mathcal{SS}(d(w,p))$ and in \cite[Theorem 3.6]{KPSTT12} that in the special case $p=1$ these ideals also coincide with the weakly compact operators.  Also they showed (\cite[Theorem 4.7]{KPSTT12}) that whenever $w\in\mathcal{W}$ satisfies condition (2SB) in Definition \ref{nuc-2sb} below, the set $\overline{\mathcal{J}}_j(d(w,p))$ is the unique immediate successor of $\mathcal{K}(d(w,p))$.  (We do not currently know whether this is the case for other choices of weights.)

Let us now explicitly define the Lorentz and Garling sequence spaces.

A sequence of positive real numbers is a {\it weight}.  We are especially interested in nonincreasing null weights which are not summable, i.e. the weights in the family  $\mathcal{W}$ defined above.  Let $\Pi$ denote the set of permutations on $\mathbb{N}$.  Then for any $1\leq p<\infty$ and any $w\in\mathcal{W}$, we define a function $\|\cdot\|_{d(w,p)}:\mathbb{K}^\mathbb{N}\to[0,\infty]$ by the rule
$$\|(a_n)_{n=1}^\infty\|_{d(w,p)}=\sup_{\sigma\in\Pi}\left(\sum_{n=1}^\infty|a_{\sigma(n)}|^pw_n\right)^{1/p}.$$
(Here, we are using $\mathbb{K}\in\{\mathbb{R},\mathbb{C}\}$, depending on whether we are working in the real or complex field.)
Usually, context allows us to abbreviate the notation $\|\cdot\|_d=\|\cdot\|_{d(w,p)}$.  We can now define the {\bf Lorentz sequence space} via
$$d(w,p)=\left\{(a_n)_{n=1}^\infty\in\mathbb{K}^\mathbb{N}:\|(a_n)_{n=1}^\infty\|_d<\infty\right\},\;\;\;w\in\mathcal{W},\;p\in[1,\infty).$$
It is well-known that $d(w,p)$ forms a Banach space under the norm $\|\cdot\|_d$, and admits a standard basis of unit vectors, usually denoted $(d_n)_{n=1}^\infty$, which is 1-symmetric.

Let $\mathbb{N}^\uparrow$ denote the set of all strictly increasing sequences of positive integers.  For $1\leq p<\infty$ and $w\in\mathcal{W}$ we can define a function $\|\cdot\|_{g(w,p)}:\mathbb{K}^\mathbb{N}\to[0,\infty]$ by the rule
$$\|(a_n)_{n=1}^\infty\|_{g(w,p)}=\sup_{(n_k)_{k=1}^\infty\in\mathbb{N}^\uparrow}\left(\sum_{k=1}^\infty|a_{n_k}|^pw_k\right)^{1/p}.$$
As expected, we write $\|\cdot\|_g=\|\cdot\|_{g(w,p)}$ when context permits.  Define the {\bf Garling sequence space} via
$$g(w,p)=\left\{(a_n)_{n=1}^\infty\in\mathbb{K}^\mathbb{N}:\|(a_n)_{n=1}^\infty\|_g<\infty\right\},\;\;\;w\in\mathcal{W},\;p\in[1,\infty).$$
It is known that $g(w,p)$ is a Banach space under the norm $\|\cdot\|_g$ whose unit vectors $(g_n)_{n=1}^\infty$ form a 1-subsymmetric basis.

As the definition of $g(w,p)$ is so similar to that of $d(w,p)$, it should not surprise us that many of the methods for finding closed ideals in the operator algebra $\mathcal{L}(d(w,p))$ carry over for $\mathcal{L}(g(w,p))$.  Let us give a list of facts analogous to those in \cite{KPSTT12} for $g(w,p)$ instead of $d(w,p)$.

\begin{theorem}Let $1\leq p<\infty$ and $w\in\mathcal{W}$.  Let $j_g:\ell_p\to g(w,p)$ denote the formal identity, i.e. the natural map between canonical bases.
\begin{itemize}\item[(i)]  $\overline{\mathcal{J}}_{\ell_p}(g(w,p))$ is a proper ideal.
\item[(ii)]  If $T\in\mathcal{L}(g(w,p))\setminus\mathcal{SS}(d(w,p))$ then $\mathcal{J}_{\ell_p}(g(w,p))\subseteq\mathcal{J}_T(g(w,p))$.  Consequently, $[\mathcal{J}_{\ell_p}+\mathcal{SS}](g(w,p))$ is the unique immediate successor of $\mathcal{SS}(g(w,p))$, and $\overline{\mathcal{J}}_{\ell_p}(g(w,p))$ is an immediate successor of $\overline{\mathcal{J}}_{\ell_p}\cap\mathcal{SS}(g(w,p))$.
\item[(iii)]  In case $p=1$ we have $\mathcal{W}(g(w,1))=\mathcal{SS}(g(w,1))$.
\item[(iv)]  $\mathcal{J}_{j_g}(g(w,p))$ is an ideal in $\mathcal{L}(g(w,p))$.
\item[(v)]  $j_g$ is class $\mathcal{FSS}$.
\item[(vi)]  $\mathcal{K}(g(w,p))\subsetneq\overline{\mathcal{J}}_{j_g}(g(w,p))$ and $\mathcal{J}_{j_g}(g(w,p))\subseteq(\mathcal{SS}\cap\mathcal{J}_{\ell_p})(g(w,p))$.
\item[(vii)]  $\mathcal{S}_{g(w,p)}(g(w,p))$ is the unique maximal ideal in $\mathcal{L}(g(w,p))$.
\end{itemize}\end{theorem}

\begin{proof}[Proof outline.]
(i)  Identical to \cite[Theorem 2.3]{KPSTT12}.  In fact, they show that $\overline{\mathcal{J}}_{\ell_p}(X)$ is a proper ideal whenever $X$ is not isomorphic to $\ell_p$.

(ii)  Identical to \cite[Theorem 3.1]{KPSTT12}.  Their proof is valid for any space $X$ such that that every infinite-dimensional subspace of $X$ contains a complemented copy of $\ell_p$, which is true for $X=g(w,p)$ by \cite[Theorem 3.1(v)]{AAW18}.

(iii)  Identical to \cite[Theorem 3.6]{KPSTT12}.  The proof is valid for any $X$ which is weakly sequentially complete and not isomorphic to $\ell_1$.  The latter condition is clear.  As $g(w,1)$ is a Banach lattice not containing $c_0$, it is weakly sequentially complete.

(iv)  Identical to \cite[Proposition 4.1]{KPSTT12}.  More generally, they show that if $X$ has an unconditional basis dominated by the $\ell_p$ basis then, letting $j_X:\ell_p\to X$ denote the formal identity, $\mathcal{J}_{j_X}(X)$ is an ideal in $\mathcal{L}(X)$.

(v)  This follows from the fact that the formal identity $j:\ell_p\to d(w,p)$ is $\mathcal{FSS}$ (\cite[Theorem 4.3]{KPSTT12}), together with the ideal property for $\mathcal{FSS}$ and the fact that ${j_g}=Id_{d(w,p),g(w,p)}\circ j$.

(vi)  This is an immediate corollary to (v).

(vii) Identical to \cite[Theorem 5.3]{KPSTT12}.  To show that $\mathcal{S}_X(X)$ is the unique maximal ideal in $\mathcal{L}(X)$, the proof requires first that $X$ has a subsymmetric weakly null basis $(x_n)_{n=1}^\infty$ such that every seminormalized block basis admits a subsequence dominating $(x_n)_{n=1}^\infty$. Obviously the unit vector basis of $g(w,p)$ is subsymmetric, and it is weakly null by \cite[Corollary 2.15]{AAW18}.  Due to \cite[Corollary 3.6]{AAW18}, the latter property holds as well.  They also require that every subspace of $X$ which is isomorphic to $X$ contains a further subspace which is complemented in $X$ and isomorphic to $X$.  The space $g(w,p)$ has this property by \cite[Theorem 3.1(ii)]{AAW18}.
\end{proof}

We can now give the following partial description of the closed ideal structure of $\mathcal{L}(g(w,p))$, for all choices of $w\in\mathcal{W}$ and all $1\leq p<\infty$.

$$
\xymatrix @R=0pt@C=12pt{ & &  &\mathcal{SS}\ar@{=>}[dr]& & & \\\{0\}\ar@{=>}[r]& \mathcal K\subsetneq\overline{\mathcal{J}}_{j_g}\ar[r]&\overline{\mathcal{J}}_{\ell_p}\cap\mathcal{SS}\ar@{:>}[dr]\ar[ur]& &[\mathcal{J}_{\ell_p}+\mathcal{SS}]\ar[r]&\mathcal{S}_{g(w,p)}\ar@{=>}[r]&\mathcal{L}\\ & & &\overline{\mathcal{J}}_{\ell_p}\ar[ur]& & & }
$$

Let us give a few final remarks relevant to the closed ideal lattice for $\mathcal{L}(d(w,p))$ and $\mathcal{L}(g(w,p))$.  In \cite{SZ15}, the authors showed that $\mathcal{L}(\ell_p\oplus\ell_q)$ admits continuum many closed ideals when $1<p<q<\infty$.  In \cite{Wa16} it was observed that their proof is valid for $\mathcal{L}(X)$ whenever $X$ contains a complemented copy of $\ell_p$, $p\in(1,2)$, and a copy (not necessarily complemented) of $\ell_q$, $q\in(p,\infty)$.  It turns out that a closer reading of their proof permits even weaker hypotheses still.  Indeed, we have the following.

\begin{theorem}\label{generalized-6}Fix any $1<p<2$ and $p<q<\infty$.  Let $X$ be a real Banach space containing a complemented copy of $\ell_p$ and a seminormalized basic sequence dominated by the canonical basis for the space $W_{q,2}$.  Let $X_\mathbb{C}$ denote the complexification of $X$.  Then $\mathcal{L}(X)$ and $\mathcal{L}(X_\mathbb{C})$ each admit a chain of closed subideals with cardinality of the continuum, which are all contained in the class $\mathcal{J}_{\ell_p}\cap\mathcal{FSS}$.\end{theorem}

The proof is almost identical to that found in \cite{SZ15}, but for completeness we shall walk through it.

We begin with some preliminary definitions.  If $n\in\mathbb{N}$, then let $(e_{p,j}^{(n)})_{j=1}^n$ denote the canonical unit vector basis for the real version of $\ell_p^n$.  For $0<w\leq 1$, we denote by $E_{p,w}^{(n)}$ the closed linear span of $(e_{p,j}^{(n)}\oplus we_{2,j}^{(n)})_{j=1}^n$ in $\ell_p^n\oplus_\infty\ell_2^n$, i.e. the space $\mathbb{R}^n$ endowed with the norm
$$\|(a_j)_{j=1}^n\|_{E_{p,w}^{(n)}}=\|(a_j)_{j=1}^n\|_{\ell_p}\vee w\|(a_j)_{j=1}^n\|_{\ell_2}.$$

\begin{proposition}[{\cite[Proposition 1]{SZ15}}]Let $1<p<2$, $0<w\leq 1$, and $n\in\mathbb{N}$, and let $(e_j^{(n)*})_{j=1}^n$ denote the biorthogonal functionals to the natural basis $(e_{p',j}^{(n)}\oplus we_{2,j}^{(n)})_{j=1}^n$ for $E_{p',w}^{(n)}$.  Then there exists a constant $K_p$, depending only on $p$ (and independent of $w$ and $n$), and a normalized 1-unconditional basic sequence $(f_j^{(n)})_{j=1}^n$ in $\ell_p^{3^n}$ satisfying
$$
\frac{1}{K_p}\left\|\sum_{j=1}^na_je_j^{(n)*}\right\|_{E_{p',w}^{(n)*}}
\leq\left\|\sum_{j=1}^na_jf_j^{(n)}\right\|_{\ell_p^{3^n}}
\leq\left\|\sum_{j=1}^na_je_j^{(n)*}\right\|_{E_{p',w}^{(n)}}.
$$
Furthermore, there exists a projection $P_{p,w}^{(n)}:\ell_p^{3^n}\to\ell_p^{3^n}$ onto the closed linear span $F_{p,w}^{(n)}$ of $(f_j^{(n)})_{j=1}^n$ such that $\|P_{p,w}^{(n)}\|\leq K_p$.\end{proposition}

If $1<p<2$ and $\textbf{w}=(w_n)_{n=1}^\infty$ is a sequence in $(0,1]$ then we define
$$
Y_{p,\textbf{w}}
=(\oplus F_{p,w_n}^{(n)})_{\ell_p}.
$$
When $p$ is clear from context, we will write $Y_{\textbf{w}}=Y_{p,\textbf{w}}$.  Note that there exists a continuous linear projection
$$
P_{\textbf{w}}=\oplus P_{p,w}^{(n)}\in\mathcal{L}(\ell_p)
$$
onto the space $Y_{\textbf{w}}$.  Finally, for the sake of following the notation of \cite{SZ15}, let us define, for each $1\leq q<\infty$, the space
$$
Z_q=W_{q,2}=(\oplus\ell_2^n)_{\ell_q}.
$$
We now prove, with nearly identical methods, a slightly different version of \cite[Theorem 6]{SZ15}.  Note that in what follows, the isomorphism between $Z_q$ and $\ell_q$ from the original proof is replaced with a bounded below operator $J$ from the span of a normalized basis dominated by $Z_q$ into an arbitrary Banach space $X$.

\begin{lemma}\label{main-lemma}Fix any $1<p<2$ and $p<q<\infty$.  Let $\emph{\textbf{v}}=(v_n)_{n=1}^\infty$ and $\emph{\textbf{w}}=(w_n)_{n=1}^\infty$ be decreasing sequences in $(0,1]$ satisfying $v_n,w_n\geq n^{-\eta}$, $n\in\mathbb{Z}^+$, $\eta:=\frac{1}{p}-\frac{1}{2}$, and also
$$
\lim_{n\to\infty}\frac{v_{\sqrt{cn}}}{w_n}=0\;\;\;\text{ for all }c\in(0,1),
$$
where we denote $v_0=1$ and $v_x=v_{\lfloor x\rfloor}$ for $x\in\mathbb{R}^+$.

Let $D$ be a real Banach space with a normalized basis which is dominated by the canonical basis for $Z_q$, and let $J:D\to X$ be an embedding of $D$ into a real Banach space $X$.  Then
$$
[\mathcal{J}^{I_{Y_{\emph{\textbf{v}}},Z_q}}](\ell_p,X)\subsetneq[\mathcal{J}^{I_{Y_{\emph{\textbf{w}}},Z_q}}](\ell_p,X).
$$
\end{lemma}

\begin{proof}
Since it was shown in the proof to \cite[Corollary 7]{SZ15} that $I_{Y_{\textbf{v}},Z_q}$ factors through $I_{Y_{\textbf{w}},Z_q}$, we just need to prove that the resulting inclusion is strict.  Note that the formal identity $I_{Z_q,D}:Z_q\to D$ is well-defined, so it is sufficient now to show that
$$
JI_{Z_q,D}I_{Y_{\textbf{w}},Z_q}P_{\textbf{w}}\notin[\mathcal{J}^{I_{Y_{\textbf{v}},Z_q}}](\ell_p,X).
$$
Set $F_n=F_{p,v_n}^{(n)}$ and $G_n=F_{p,w_n}^{(n)}$ with respective canonical bases $(f_j^{(n)})_{j=1}^n\subset\ell_p$ and $(g_j^{(n)})_{j=1}^n\subset\ell_p$.  Write $((d_j^{(n)})_{j=1}^n)_{n=1}^\infty$ for the basis of $D$, and let $((x_j^{(n)*})_{j=1}^n)_{n=1}^\infty$ denote the biorthogonal functionals in $X^*$ to the basic sequence $((Jd_j^{(n)})_{j=1}^n)_{n=1}^\infty$ in $X$.  Note that $((x_j^{(n)*})_{j=1}^n)_{n=1}^\infty$ is uniformly bounded by some $K\in[1,\infty)$.  Hence, we can define uniformly bounded functionals $\Phi_m\in\mathcal{L}(\ell_p,X)^*$, $m\in\mathbb{N}$, by the rule
$$
\Phi_m(V)=\frac{1}{m}\sum_{i=1}^mx_i^{(m)*}(Vg_i^{(m)}),\;\;\;V\in\mathcal{L}(\ell_p,X).
$$
By weak* compactness of the closed unit ball of a dual space, we can find a weak* accumulation point $\Phi\in\mathcal{L}(\ell_p,X)^*$ of $(\Phi_m)_{m=1}^\infty$.

Before continuing, we need to explicitly state following lemma, which follows from the proof to \cite[Theorem 6]{SZ15}.

\begin{lemma}\label{accumulation}
For each $m\in\mathbb{N}$, let $B_m:G_m\to Y_{\emph{\textbf{v}}}$ denote a continuous linear operator of norm $\leq 1$.  Then
$$
\limsup_{m\to\infty}\frac{1}{m}\sum_{i=1}^m\left\|I_{Y_{\emph{\textbf{v}}},Z_q}B_mg_i^{(m)}\right\|_{Z_q}=0.
$$
\end{lemma}

Now fix any $A\in\mathcal{L}(Z_q,X)$ and $B\in\mathcal{L}(\ell_p,Y_{\textbf{v}})$ with $\|A\|,\|B\|\leq 1$, and for each $m\in\mathbb{N}$, let $B_m=B|_{G_m}:G_m\to Y_{\textbf{v}}$ denote the restriction of $B$ to $G_m$.  Then by Lemma \ref{accumulation} we have
\begin{align*}
\limsup_{m\to\infty}\left|\Phi_m(AI_{Y_{\textbf{v}},Z_q}B)\right|
&\leq\limsup_{m\to\infty}\frac{1}{m}\sum_{i=1}^m\left|(A^*x_i^{(m)*})(I_{Y,Z_q}B_mg_i^{(m)})\right|
\\&\leq K\|A\|\limsup_{m\to\infty}\frac{1}{m}\sum_{i=1}^m\left\|I_{Y,Z_q}B_mg_i^{(m)}\right\|_{Z_q}
\\&=0.
\end{align*}
It follows that $\Phi$ annihilates $\mathcal{J}^{I_{Y_{\textbf{v}},Z_q}}(\ell_p,X)$ and hence also $[\mathcal{J}^{I_{Y_{\textbf{v}},Z_q}}](\ell_p,X)$, whereas clearly $\Phi(JI_{Z_q,D}I_{Y_{\textbf{w}},Z_q}P_{\textbf{w}})=1$.
\end{proof}

Now we are ready to prove our generalization of the main result in \cite{SZ15}.

\begin{proof}[Proof of Theorem \ref{generalized-6}.]
As in \cite[\S2.4 \& \S2.5]{SZ15}, we fix any chain $\mathcal{C}$ of subsets of $\mathbb{N}$, with cardinality of the continuum, such that if $M,N\in\mathcal{C}$ then either $N\subseteq M$ and $|M\setminus N|=\infty$, or else $M\subseteq N$ and $|N\setminus M|=\infty$.  It was shown in the proof of \cite[Theorem A]{SZ15} that there exists a chain $(\textbf{w}_M)_{M\in\mathcal{C}}$ of decreasing sequences in $(0,1]$, such that whenever $M,N\in\mathcal{C}$ then, assuming without loss of generality $M\subseteq N$, the conditions of Lemma \ref{main-lemma} are satisfied for $\textbf{v}=\textbf{w}_N$ and $\textbf{w}=\textbf{w}_M$.  Hence,
$$
[\mathcal{J}^{I_{\textbf{v},Z_q}}](\ell_p,X)
\subsetneq[\mathcal{J}^{I_{\textbf{w},Z_q}}](\ell_p,X).
$$
By \cite[Proposition 2.2]{Wa16}, this also gives us
$$
[\mathcal{J}^{I_{\textbf{v},Z_q}}](\ell_p,X_{\mathbb{C}})
\subsetneq[\mathcal{J}^{I_{\textbf{w},Z_q}}](\ell_p,X_\mathbb{C}).
$$
Finally, by \cite[Proposition 8]{SZ15}, the real version of the map $I_{\textbf{w},Z_q}$ is $\mathcal{FSS}$, and this extends to the complexified version by \cite[Proposition 2.3]{Wa16}.
\end{proof}

Although not every Lorentz or Garling sequence space of the form $d(w,p)$ or $g(w,p)$, $p\in(1,2)$, contains a seminormalized basic sequence dominated by the $W_{q,2}$ basis, $q\in(p,\infty)$, we can construct certain weights so that it does.

\begin{proposition}\label{lq-dominates-lorentz}If $1\leq p<q<\infty$ and $w=(w_i)_{i=1}^\infty\in\mathcal{W}\cap\ell_{q/(q-p)}$ then the canonical basis for $d(w,p)$ is $C$-dominated by the canonical basis for $\ell_q$, where $C=\|(w_i)_{i=1}^\infty\|_{q/(q-p)}^{1/p}$.\end{proposition}

\begin{proof}Note that $q/p$ and $q/(q-p)$ are conjugate, i.e.
$$\frac{1}{q/p}+\frac{1}{q/(q-p)}=1.$$
Thus, denoting $(d_i)_{i=1}^\infty$ the canonical basis of $d(w,p)$, and $(\hat{a})_{i=1}^\infty$ the nonincreasing rearrangement of $(|a_i|)_{i=1}^\infty$ for some arbitrary $(a_i)_{i=1}^\infty\in c_{00}$ (where $c_{00}$ denotes the space of finitely-supported scalar sequences), by H\"older we have
\begin{align*}
\|\sum_{i=1}^na_id_i\|_d
&=\left(\sum_{i=1}^n\hat{a}_i^pw_i\right)^{1/p}
\\\\&\leq\|(w_i)_{i=1}^\infty\|_{q/(q-p)}^{1/p}\|(\hat{a}_i^p)_{i=1}^n\|_{q/p}^{1/p}
\\\\&=C\|(a_i)_{i=1}^n\|_q.
\end{align*}\end{proof}

\begin{corollary}\label{main-1}If $1<p<2$ and $w\in\mathcal{W}\cap\ell_{2/(2-p)}$ then $\mathcal{L}(d(w,p))$ admits a chain of closed ideals, with cardinality of the continuum, which are all class $\mathcal{J}_{\ell_p}\cap\mathcal{FSS}$.  The same is true of $\mathcal{L}(g(w,p))$.\end{corollary}

\begin{proof}
By Proposition \ref{lq-dominates-lorentz}, the canonical basis for $d(w,p)$, and hence also for $g(w,p)$, are each dominated by the $\ell_2=W_{2,2}$ basis.  Let's note two more things:  First, the complex versions of $d(w,p)$ and $g(w,p)$ coincide with the respective complexifications of the real versions; second, $d(w,p)$ and $g(w,p)$ each contain complemented copies of $\ell_p$ (cf. \cite[Proposition 4]{LT72} for the Lorentz case and \cite[Theorem 3.1(v)]{AAW18} for the Garling case).  Now apply Theorem \ref{generalized-6}.\end{proof}

Unfortunately, the proof of Corollary \ref{main-1} is invalid for other choices of $p$.  Indeed, the proof of Theorem \ref{generalized-6} uses the fact that $\ell_p$ is isomorphic to $W_{p,2}$ for $1<p<\infty$, whereas this is not true when $p=1$.  It also used auxiliary operators whose constructions depend on the condition that $p<2$.

\section{$\tand{p}$ is isomorphic to $W_p$, $1\leq p<\infty$}

For $1<p\leq\infty$, we define the {\bf $\boldsymbol{p}$th Ces\'{a}ro sequence space} $\ces{p}$ as the Banach space of all scalar sequences $(a_i)_{i=1}^\infty\in\mathbb{K}^\mathbb{N}$ satisfying
$$\|(a_i)_{i=1}^\infty\|_{\ces{p}}<\infty,$$
where we set
$$\|(a_i)_{i=1}^\infty\|_{\ces{p}}=\|\left(\frac{1}{i}\sum_{j=1}^i|a_j|\right)_{i=1}^\infty\|_p=\left\{\begin{array}{ll}\displaystyle\left(\sum_{i=1}^\infty\left[\frac{1}{i}\sum_{j=1}^i|a_j|\right]^p\right)^{1/p}&\text{ if }1<p<\infty\\\\\displaystyle\sup_{i\in\mathbb{N}}\frac{1}{i}\sum_{j=1}^i|a_j|&\text{ if }p=\infty.\end{array}\right.$$
If $1\leq p<\infty$ then we define the {\bf $\boldsymbol{p}$th Tandori sequence space} $\tand{p}$ as the Banach space of all scalar sequences $(a_i)_{i=1}^\infty\in\mathbb{K}^\mathbb{N}$ satifying
$$\|(a_i)_{i=1}^\infty\|_{\tand{p}}<\infty,$$
where
$$\|(a_i)_{i=1}^\infty\|_{\tand{p}}=\|\left(\sup_{j\geq i}|a_j|\right)_{i=1}^\infty\|_p=\left(\sum_{i=1}^\infty\sup_{j\geq i}|a_j|^p\right)^{1/p}.$$
Note that if we tried to extend these definitions in the natural way to all $1\leq p\leq\infty$ then we would have $\ces{1}=\{0\}$ and $\tand{\infty}=\ell_\infty$.  Thus, we limit $\ces{p}$ to $1<p\leq\infty$ and $\tand{p}$ to $1\leq p<\infty$.

Denote by $(e_k)_{k=1}^\infty$ the unit vectors in $c_{00}$.  These vectors are not normalized in either $\tand{p}$ or $\ces{p}$, and so instead we consider vectors
$$f_k=k^{-1/p}e_k\;\;\;\text{ and }\;\;\; g_k=k^{1/p}e_k,\;\;\;k\in\mathbb{N}.$$
It is known that $(f_k)_{k=1}^\infty$ forms a normalized 1-unconditional basis for $\tand{p}$ with respective seminormalized (but not normalized) coordinate functionals $(g_k)_{k=1}^\infty\subset\ces{p'}$, $\frac{1}{p}+\frac{1}{p'}=1$.  Bennett (\cite{Be96}) has shown that $\ces{p'}\approx(\tand{p})^*$ (isomorphic but not isometric), and that when $1<p<\infty$ both spaces are reflexive so that the functionals $(g_k)_{k=1}^\infty$ form a seminormalized basis for $\ces{p'}$.  However, the space $\ces{\infty}$ is nonseparable, so we shall instead write $ces_0(\infty)=[g_k]_{k=1}^\infty$.  In this case, $ces_0(\infty)^*\approx\tand{1}$ and $ces_0(\infty)^{**}\approx\tand{1}^*\approx\ces{\infty}$.

The {\bf unconditional constant} $K$ of an unconditional basis $(x_n)_{n=1}^N$, $1\leq N\leq\infty$, for a Banach space $X$ is defined by
\[K=\sup_{s\in\mathcal{S}}\|M_s\|,\]
where $\mathcal{S}$ is the compact metric space of sequences $s=(s_n)_{n=1}^N$ of signs $s_n\in\{\pm 1\}$, and each $M_s\in\mathcal{L}(X)$, $s\in\mathcal{S}$, is defined by $M_s\sum_{n=1}^Na_nx_n=\sum_{n=1}^Ns_na_nx_n$.  It is well-known that, in this case, $K<\infty$.

A basis $(x_n)_{n=1}^\infty$ is called {\bf $\boldsymbol{C}$-seminormalized}, $C\in[1,\infty)$, whenever $$\frac{1}{C}\leq\|x_n\|\leq C$$ for all $n\in\mathbb{N}$.

\begin{proposition}\label{uniformly-equivalent}Let $C,K,M\in(0,\infty)$ and $N\in\mathbb{N}$.  Suppose $(x_n)_{n=1}^N$ be a $K$-unconditional and $C$-seminormalized basis for an $N$-dimensional Banach space $X$ which satisfies $\|x_1+\cdots+x_N\|\leq M$.  Then $(x_n)_{n=1}^N$ is $(2CMK^2)$-equivalent to the canonical basis of $\ell_\infty^N$.\end{proposition}

\begin{proof}Fix $(a_n)_{n=1}^N\in\mathbb{K}^N$, and select $j\in\{1,\cdots,N\}$ satisfying $|a_j|=\|(a_n)_{n=1}^N\|_\infty$.  Then, using well-known inequalities (cf., e.g., \cite[Lemma 1.49 and subsequent remarks]{AA02}),
\[\|(a_n)_{n=1}^N\|_\infty=|a_j|\leq C\|a_jx_j\|_X\leq CK\|\sum_{n=1}^Na_nx_n\|_X\leq 2CK^2M\|(a_n)_{n=1}^N\|_\infty.\]\end{proof}

\begin{remark}\label{uniformly-equivalent-remark}In the above inequalities, the appearance of constant 2 is only required in the complex setting.  In case $X$ is a real Banach space, it can be ignored to obtain $(K^2CM)$-equivalence to the $\ell_\infty^N$ basis.  In case $X$ is a complex Banach space and $(x_n)_{n=1}^N$ is normalized, as long as $\|\sum_{n=1}^Na_nx_n\|_X=\|\sum_{n=1}^N\epsilon_na_nx_n\|$ for all $(\epsilon_n)_{n=1}^N\in\mathbb{T}^N$, where $\mathbb{T}$ is the complex unit circle, the basis $(x_n)_{n=1}^N$ is $M$-equivalent to the $\ell_\infty^N$ basis.  In particular, if $(f_k)_{k=1}^\infty$ is the canonical normalized basis for $\tand{p}$, $1\leq p<\infty$, and $\|f_m+\cdots+f_n\|_{\tand{p}}\leq M$ for some $1\leq m\leq n<\infty$, then
\[\ell_\infty^{n-m+1}\lesssim_1(f_k)_{k=m}^n\lesssim_M\ell_\infty^{n-m+1}.\]\end{remark}

For each $1\leq p<\infty$, write
$$V_p=\left(\bigoplus_{n=0}^\infty\ell_\infty^{2^n}\right)_{\ell_p}=\left(\bigoplus_{n=1}^\infty\ell_\infty^{2^{n-1}}\right)_{\ell_p}.$$
and denote by $(v_n^{(p)})_{n=1}^\infty$ its canonical basis.  Let us also use the notation $(f_n^{(p)})_{n=1}^\infty$ for the canonical basis of $\tand{p}$.  In the proof of \cite[Corollary 3]{ALM18}, the authors constructed the following.

\begin{theorem}[{\cite{ALM18}}]\label{ALM18} There exists a linear isomorphism $T:\tand{1}\to V_1$ defined by the rule $Tf_n^{(1)}=v_n^{(1)}$, and satisfying $\|T\|\leq 72$ and $\|T^{-1}\|\leq 2$.\end{theorem}

In fact, this extends easily to all $1\leq p<\infty$.

\begin{corollary}\label{main-2}
Let $1\leq p<\infty$.  Then
\[\tand{p}\approx\left(\bigoplus_{n=1}^\infty\ell_\infty^{2^{n-1}}\right)_{\ell_p}.\]
In particular, the canonical basis for $\tand{p}$ is $2^{1/p}$-dominated by and $72^{1/p}$-dominates the canonical basis for $(\oplus_{n=1}^\infty\ell_\infty^{2^{n-1}})_{\ell_p}$.
\end{corollary}

\begin{proof}
The case $p=1$ is immediate from Theorem \ref{ALM18}.  For the case $1<p<\infty$, observe that, for $(a_n)_{n=1}^\infty\in c_{00}$,
\begin{align*}
\left\|\sum_{n=1}^\infty a_nf_n\right\|_{\tand{p}}^p
&=\sum_{n=1}^\infty\sup_{j\geq n}\left(j^{-1/p}|a_j|\right)^p
\\&=\sum_{n=1}^\infty\sup_{j\geq n}\left(j^{-1}|a_j|^p\right)
\\&=\left\|\sum_{n=1}^\infty|a_n|^pf_n^{(1)}\right\|_{\tand{1}}
\\&\leq 2\left\|\sum_{n=1}^\infty|a_n|^pv_n^{(1)}\right\|_{V_1}
\\&=2\sum_{n=1}^\infty\sup_{2^{n-1}\leq j\leq 2^n-1}|a_j|^p
\\&=2\sum_{n=1}^\infty\left(\sup_{2^{n-1}\leq j\leq 2^n-1}|a_j|\right)^p
\\&=2\left\|\sum_{n=1}^\infty a_nv_n^{(p)}\right\|_{V_p}^p
\end{align*}
The reverse inequality clearly also holds with constant 72 instead of 2.
\end{proof}

\begin{remark}As $(\oplus_{n=1}^\infty\ell_\infty^{2^{n-1}})_p$ is isomorphic to $W_p$, it follows that $\tand{p}\approx W_p$ and $\ces{p'}\approx W_p^*=W_{p',1}$, although not in the natural ways.\end{remark}

For the remainder of this section, let's denote by $((e_k^{(n)})_{k=1}^\infty)_{n=1}^\infty$ the canonical basis for $\ell_p(c_0)$, $1\leq p<\infty$, where each $(e_k^{(n)})_{k=1}^\infty$ is a copy of the unit vector basis for $c_0$.  We can now give an application of Corollary \ref{main-2}.

\begin{proposition}Let $1\leq p<\infty$, and let $(f_k)_{k=1}^\infty$ be the canonical normalized basis for $\tand{p}$.  Then $\tand{p}$ is isometric to a subspace of $\ell_p(c_0)$ via the mapping
\[f_k\mapsto\frac{1}{k^{1/p}}\sum_{n=1}^ke_k^{(n)},\;\;\;k\in\mathbb{N},\]
where $((e_k^{(n)})_{k=1}^\infty)_{n=1}^\infty$ is the canonical basis for $\ell_p(c_0)$.\end{proposition}

\begin{proof}Observe that the map
\[(a_k)_{k=1}^\infty\in\tand{p}\mapsto\left((a_k)_{k=n}^\infty\right)_{n=1}^\infty\in\ell_p(c_0)\]
is a linear isometry.  Hence, so is the map
\[(k^{-1/p}a_k)_{k=1}^\infty\in\tand{p}\mapsto\left((k^{-1/p}a_k)_{k=n}^\infty\right)_{n=1}^\infty\in\ell_p(c_0).\]
However, this is just the linear map defined by
\[f_k\mapsto\frac{1}{k^{1/p}}\sum_{n=1}^ke_k^{(n)}.\]
\end{proof}

\begin{corollary}Let $1\leq p<\infty$, and let $((e_k^{(n)})_{k=1}^\infty)_{n=1}^\infty$ denote the canonical basis for $\ell_p(c_0)$.  Then the sequence formed by
\[\frac{1}{k^{1/p}}\sum_{n=1}^ke_k^{(n)},\;\;\;k\in\mathbb{N}\]
is a normalized basic sequence equivalent to the canonical basis for $(\oplus_{n=1}^\infty\ell_\infty^{2^{n-1}})_p$.\end{corollary}

\section{$W_p$ is complemented in $d(w,p)$ and $g(w,p)$ for certain $w$}

In this section we show that $W_p$ is complemented in $d(w,p)$ and $g(w,p)$ for $1\leq p<\infty$ and for certain choices of $w$.  As an application, we give a new, slightly different proof of a result already known from \cite{LLR04} regarding the closed ideal structure of $\mathcal{L}(W_p)$.

\begin{definition}\label{nuc-2sb}
Let $w=(w_i)_{i=1}^\infty\in\mathcal{W}$.

We say that $w$ is {\bf(NUC)} whenever
$$\inf_{n\in\mathbb{N}}\frac{\sum_{i=1}^{2n}w_i}{\sum_{i=1}^nw_i}=1.$$

We say that $w$ is {\bf(2SB)} whenever $$\sup_{n,k\in\mathbb{N}}\frac{\sum_{i=1}^{nk}w_i}{(\sum_{i=1}^nw_i)(\sum_{i=1}^kw_i)}<\infty.$$
\end{definition}

Condition (NUC) was originally studied in \cite{Al75} for the cases $1<p<\infty$ to characterize reflexive but non-uniformly convex Lorentz sequence spaces.  Condition (2SB) was studied in \cite{ACL73} to characterize those Lorentz sequence spaces admitting precisely two symmetric basic sequences.

\begin{remark}There are many such weights simultaneously satisfying both (NUC) and (2SB); perhaps the ``nicest'' example is formed by setting $w_i=i^{-1}$ for each $i\in\mathbb{N}$.\end{remark}

Let us begin with the main result of this section, before proceeding to some of its applications.

\begin{theorem}\label{complemented-copies}
Let $1\leq p<\infty$ and suppose $w\in\mathcal{W}$ is (NUC).  Then $g(w,p)$ admits a 2-complemented copy of $W_p$ and $d(w,p)$ admits a 1-complemented copy of $W_p$.
\end{theorem}

\begin{proof}
It was shown in \cite[Lemmas 4 and 5]{Al75} that if $w$ is (NUC) then $d(w,p)$ contains uniform copies of $\ell_\infty^n$ spanned by normalized constant coefficient block bases $(v_i^{(k)})_{i=1}^n$ with coefficients tending to zero, and their proof is valid also for $g(w,p)$.  More specifically, we have
$$v_i^{(k)}=c_k^{-1}\sum_{j\in A_i^{(k)}}d_j,\;\;\;c_k=\left\|\sum_{j\in A_i^{(k)}}d_j\right\|_{d(w,p)},\;\;\;i=1,\cdots,k,$$
for successive finite subsets of $\mathbb{N}$ which we label
$$A_1^{(k)}<A_2^{(k)}<\cdots<A_k^{(k)}<A_1^{(k+1)}<\cdots,\;\;\;k\in\mathbb{N},$$
with
$$M_k:=\#A_1^{(k)}=\#A_2^{(k)}=\cdots=\#A_k^{(k)},\;\;\;k\in\mathbb{N}.$$
For convenience, write
$$u_i^{(k)}=\sum_{j\in A_i^{(k)}}d_j,\;\;\;i=1,\cdots,k,\;\;\;k\in\mathbb{N},$$
so that each $v_i^{(k)}=c_k^{-1}u_i^{(k)}$.  In particular, the author showed that for each $K\in(1,\infty)$ one can choose the above so that for each $k\in\mathbb{N}$, the finite basic sequence $(v_i^{(k)})_{i=1}^k$ 1-dominates and is $K$-dominated by the unit vector basis of $\ell_\infty^k$.

In the proof of \cite[4.e.3]{LT77} the authors find $C\in(1,\infty)$ and $n_1<n_2<\cdots$ such that
$$
\frac{1}{C}\left(\sum_{k=1}^\infty|a_k|^p\right)^{1/p}
\leq\left\|\sum_{k=1}^\infty a_kv_1^{(n_k)}\right\|_{d(w,p)}
\leq\left(\sum_{k=1}^\infty|a_k|^p\right)^{1/p}.
$$
An analogous inequality is shown in \cite[Theorem 3.3]{AAW18} to hold for $g(w,p)$.  For convenience, we pass to subsequences and relabel $k=n_k$.  By 1-subsymmetry we now have
$$
(v_{i_k}^{(k)})_{k=1}^\infty\approx_C\ell_p
$$
for any choices $i_k\in\{1,\cdots,k\}$, $k\in\mathbb{N}$.

Now, pick any $((a_i^{(k)})_{i=1}^k)_{k=1}^\infty\in c_{00}$ and for each $k\in\mathbb{N}$ we choose $i_k\in\{1,\cdots,k\}$ so that
$$\left|a_{i_k}^{(k)}\right|=\sup_{1\leq i\leq k}\left|a_i^{(k)}\right|.$$
In the Lorentz case, there exists a permutation $\sigma$ of $\mathbb{N}$, and there exist intervals $I_1<I_2<\cdots$ with $\#I_k=\sigma(k)\cdot M_{\sigma(k)}$ for each $k\in\mathbb{N}$ and $\bigcup_{k=1}^\infty I_k=\mathbb{N}$ such that
\begin{align*}
\left\|\sum_{k=1}^\infty\sum_{i=1}^k a_i^{(k)}v_i^{(k)}\right\|_{d(w,p)}
&\leq\left\|\sum_{k=1}^\infty\left|a_{i_k}^{(k)}\right|\sum_{i=1}^kv_i^{(k)}\right\|_{d(w,p)}
\\&=\left\|\sum_{k=1}^\infty\left|a_{i_k}^{(k)}\right|c_k^{-1}\sum_{i=1}^k\sum_{j\in A_i^{(k)}}d_i\right\|_{d(w,p)}
\\&=\left(\sum_{k=1}^\infty\left|a_{i_{\sigma(k)}}^{(\sigma(k))}\right|^pc_{\sigma(k)}^{-p}\sum_{i\in I_k}w_i\right)^{1/p}
\\&\leq\left(\sum_{k=1}^\infty\left|a_{i_{\sigma(k)}}^{(\sigma(k))}\right|^pc_{\sigma(k)}^{-p}\sum_{i=1}^{\sigma(k)M_{\sigma(k)}}w_i\right)^{1/p}
\\&=\left(\sum_{k=1}^\infty\left|a_{i_{\sigma(k)}}^{(\sigma(k))}\right|^pc_{\sigma(k)}^{-p}\left\|\sum_{i=1}^{\sigma(k)}u_i^{(\sigma(k))}\right\|_{d(w,p)}^p\right)^{1/p}
\\&=\left(\sum_{k=1}^\infty\left|a_{i_{\sigma(k)}}^{(\sigma(k))}\right|^p\left\|\sum_{i=1}^{\sigma(k)}v_i^{(\sigma(k))}\right\|_{d(w,p)}^p\right)^{1/p}
\\&\leq K\left(\sum_{k=1}^\infty\left|a_{i_{\sigma(k)}}^{(\sigma(k))}\right|^p\right)^{1/p}
\\&=K\left\|((a_i^{(k)})_{i=1}^k)_{k=1}^\infty\right\|_{W_p}.
\end{align*}
On the other hand, by 1-unconditionality,
\begin{align*}
\left\|\sum_{k=1}^\infty\sum_{i=1}^k a_i^{(k)}v_i^{(k)}\right\|_{d(w,p)}
&\geq\left\|\sum_{k=1}^\infty a_{i_k}^{(k)}v_{i_k}^{(k)}\right\|_{d(w,p)}
\\&\geq\frac{1}{C}\left\|(a_{i_k})_{k=1}^\infty\right\|_{\ell_p}
\\&=\frac{1}{C}\left\|\left(\left(a_i^{(k)}\right)_{i=1}^k\right)_{k=1}^\infty\right\|_{W_p}.
\end{align*}
Since $K$ can be made arbitrarily close to 1, it follows that the constant-coefficient block basic sequence $((v_i^{(k)})_{i=1}^k)_{k=1}^\infty$ is $C$-equivalent to the standard $W_p$ basis.  Since the blocks are formed by constant coefficients with respect to a 1-symmetric (resp. 1-subsymmetric) basis, it follows by \cite[3.a.4]{LT77} that $W_p$ is complemented in $d(w,p)$ (resp. $g(w,p)$) with constant 1 (resp. 2).
\end{proof}

\begin{remark}
Even though the copy of $W_p$, $1\leq p<\infty$, from Theorem \ref{complemented-copies} above is 1-complemented in $d(w,p)$ when $w\in\mathcal{W}$ is (NUC), the copy is not itself isometric.  That is, we can find $C\in(1,\infty)$ such that there exists a norm-1 projection $P\in\mathcal{L}(d(w,p))$ such that the range of $P$ is $C$-isomorphic (not isometric) to $W_p$.
\end{remark}

Let us take a moment to prove that the reverse inclusion does not hold, and indeed that neither $d(w,p)$ nor $g(w,p)$ are ever subspaces of $W_p$, for any $1\leq p<\infty$ or $w\in\mathcal{W}$.

\begin{proposition}\label{fdd-block-sequence}
Let $E:=(\oplus_{n=1}^\infty E_n)_p$ be the $\ell_p$-sum of finite-dimensional spaces $(E_n)_{n=1}^\infty$.  If $(z_i)_{i=1}^\infty$ is a $C$-seminormalized FDD-block sequence then it is $C$-equivalent to $\ell_p$.\end{proposition}

\begin{proof}
As $(z_i)_{i=1}^\infty$ is a FDD-block sequence, we can find $k_1<k_2<\cdots$ and vectors $e_k\in E_k$, such that
$$z_i=\sum_{k=k_i}^{k_{i+1}-1}e_k,\;\;\;i\in\mathbb{N}.$$
Then
$$\left\|\sum_{i=1}^\infty c_iz_i\right\|_E=\left(\sum_{i=1}^\infty |c_i|^p\sum_{k=k_i}^{k_{i+1}-1}\|e_k\|_{E_k}^p\right)^{1/p}\in[C^{-1},C]\cdot\|(c_i)_{i=1}^\infty\|_p.$$
\end{proof}

\begin{corollary}\label{not-subspace}
Let $1\leq p<\infty$ and $w\in\mathcal{W}$.  Then $d(w,p)$ is not a subspace of $E:=(\oplus_{n=1}^\infty E_n)_p$ for any $\ell_p$-sum of finite-dimensional spaces.  Neither is $g(w,p)$.
\end{corollary}

\begin{proof}
If $d(w,p)$ or $g(w,p)$ is a subspace, we can find $(z_i)_{i=1}^\infty$ a seminormalized basic sequence in $E$ which is equivalent to the $d(w,p)$ or $g(w,p)$ canonical basis.  As the latter is weakly null (cf., e.g., \cite[Remark 1.7]{KPSTT12}) and subsymmetric, we may assume $(z_i)_{i=1}^\infty$ is block basic, and that its blocks do not overlap in any space $E_n$. However, such blocks are equivalent to the $\ell_p$ basis by Lemma \ref{fdd-block-sequence}. This is a contradiction.
\end{proof}

\begin{corollary}
Let $w\in\mathcal{W}$ be (NUC) and $1\leq p<\infty$, and let $P\in\mathcal{L}(d(w,p))$ be a projection whose range is isomorphic to $W_p$.  Then $P$ fails to fix a copy of $d(w,p)$, and hence, for any Banach space $X$,
$$\overline{\mathcal{J}}_{\ell_p}(X)\subseteq\overline{\mathcal{J}}_P(X)\subseteq\mathcal{S}_{d(w,p)}(X).$$
The same goes for $g(w,p)$.
\end{corollary}

As a final aside, we note an interesting fact about the duals of Lorentz and Garling sequence spaces for the case when $w$ is (NUC) and $p=1$.

\begin{corollary}
Let $w\in\mathcal{W}$ be (NUC).  Then each of $g(w,1)^*$ and $d(w,1)^*$ contain complemented copies of both $L_1$ and $L_\infty$.  Consequently, their operator algebras $\mathcal{L}(d(w,1)^*)$ or $\mathcal{L}(g(w,1)^*)$ each admit uncountably many closed ideals.
\end{corollary}

\begin{proof}
Observe that $W_1^*$ contains a complemented copy of $\ell_\infty(\approx L_\infty)$. On the other hand, it was shown in \cite[Theorem 1]{HS73} that $W_1^*$ also contains a complemented copy of $L_1$.  By Theorem \ref{complemented-copies}, $d(w,1)^*$ and $g(w,1)^*$ each contain a complemented copy of $W_1^*$.  It was shown in \cite[Theorem 3.7 and Proposition 3.8]{SW15} that if $Z$ is a Banach space containing complemented copies of $\ell_1$ and $\ell_\infty$ then $\mathcal{L}(Z)$ admits an uncountable chain of closed ideals.  As $\ell_1$ is complemented in $L_1$, this completes the proof.
\end{proof}

\begin{remark}We still don't know whether $\mathcal{L}(d(w,1))$ or $\mathcal{L}(g(w,1))$ themselves admit infinitely many closed ideals.\end{remark}

We close by giving some final remarks on the closed ideals in $\mathcal{L}(W_p)$. In \cite{Le15} was shown that $\mathcal{L}(W_1)$ admits a unique maximal ideal characterized by the set
$$\mathcal{M}_1:=\left\{T\in\mathcal{L}(W_1):\theta T\in\overline{\mathcal{J}}_{\ell_1}(W_1,\ell_\infty)\text{ for some bdd. below }\theta\in\mathcal{L}(W_1,\ell_\infty)\right\}.$$
Note that $\mathcal{M}_1$ is the injective hull of the ideal $\overline{\mathcal{J}}_{\ell_1}(W_1)$.  Analogously, it was shown in \cite[Theorem 3.2]{LOSZ12} that the surjective hull of $\overline{\mathcal{J}}_{c_0}(W_{0,1})$ is the unique maximal ideal of $\mathcal{L}(W_{0,1})$.  Then, in \cite{KL16} that, the authors showed that the set
$$\mathcal{M}_p:=\left\{T\in\mathcal{L}(W_p):T\text{ does not fix }\ell_\infty^n\text{'s uniformly}\right\}.$$
forms the unique maximal ideal in $\mathcal{L}(W_p)$ when $1<p<\infty$. Thus, to completely classify the closed ideals in $\mathcal{L}(W_p)$, $1\leq p<\infty$, it is sufficient to prove that $\mathcal{M}_p$ is an immediate successor to $\mathcal{K}(W_p)$.  Indeed, considering the complete descriptions of the lattices of closed ideals for $\mathcal{L}(W_{2,0})$ and $\mathcal{L}(W_{2,1})$ proved in \cite{LLR04,LSZ06} (see \S1 above), this seems likely.  Unfortunately, we don't currently know whether it is the case.  In particular, we do not know of any $p\in[1,\infty)$ such that $\mathcal{M}_p=\overline{\mathcal{J}}_{\ell_p}(W_p)$.

\section*{Acknowledgement}

The author thanks an anonymous referee for his valuable suggestions on improving this paper.

\bibliographystyle{amsplain}

\end{document}